\documentclass{gtpart}
\usepackage{amssymb}
\usepackage{amsthm}
\usepackage{bbm}
\usepackage{amsmath,amscd}
\usepackage[mathscr]{euscript}
\usepackage[all]{xy}
\usepackage[usenames,dvipsnames]{color}
\usepackage[utf8]{inputenc}
\usepackage{lmodern}
\usepackage[T1]{fontenc}
\usepackage{exscale}
\usepackage[textwidth=14.5cm,hcentering]{geometry}
\title{Rigidification of higher categorical structures}
\author{Giovanni Caviglia}
\givenname{Giovanni}
\surname{Caviglia}
\address{Radboud Universiteit Nijmegen, Institute for Mathematics, Astrophysics, and Particle Physics\\ Heyendaalseweg 135\\ 6525 AJ Nijmegen\\ the Netherlands}
\email{g.caviglia@math.ru.nl}
\urladdr{http://www.math.ru.nl/~gcaviglia/}
\author{Geoffroy Horel}
\givenname{Geoffroy}
\surname{Horel}
\address{Max Planck Institute for Mathematics\\
Vivatsgasse 7\\
53111 Bonn\\
Deutschland}
\email{geoffroy.horel@gmail.com}
\urladdr{http://geoffroy.horel.org/}

\keywords{Internal operads, internal $n$-categories, limit sketches, model categories}

\subject{primary}{msc2010}{18C30}
\subject{primary}{msc2010}{18D35}
\subject{primary}{msc2010}{55U35}
\subject{secondary}{msc2010}{18D05}
\subject{secondary}{msc2010}{18D50}

\arxivreference{1511.01119}
\arxivpassword{nknbv}

\volumenumber{}
\issuenumber{}
\publicationyear{}
\papernumber{}
\startpage{}
\endpage{}
\doi{}
\MR{}
\Zbl{}
\received{}
\revised{}
\accepted{}
\published{}
\publishedonline{}
\proposed{}
\seconded{}
\corresponding{}
\editor{}
\version{}
\newtheorem{theo}{Theorem}[section]
\newtheorem{lemm}[theo]{Lemma}
\newtheorem{prop}[theo]{Proposition}
\newtheorem{coro}[theo]{Corollary}
\newtheorem*{theo*}{Theorem}

\theoremstyle{definition}
\newtheorem{defi}[theo]{Definition}

\newtheorem{assu}[theo]{Assumption}

\newtheorem{example}[theo]{Example}

\newtheorem{rem}[theo]{Remark}

\newcommand{\op}{^{\mathrm{op}}}
\newcommand{\cat}{\mathbf}
\newcommand{\Set}{\cat{Set}}
\newcommand{\Cat}{\cat{Cat}}
\newcommand{\id}{\mathrm{id}}
\renewcommand{\R}{\mathbb{R}}
\renewcommand{\L}{\mathbb{L}}
\newcommand{\Map}{\operatorname{Map}}

\newcommand{\Fun}{\operatorname{Fun}}
\newcommand{\sSet}{\mathrm{\Set}^{\Delta\op}}

\renewcommand{\S}{\cat{S}}
\newcommand{\X}{\cat{X}}

\newcommand{\diagcs}[9]{\xymatrix@C=#9cm{#1 \ar[d]_{#5} \ar[r]^-{#6} &  #2 \ar[d]^-{#7}\\
          #3 \ar[r]_-{#8} & #4}}

\newcommand{\Fcat}[2]{\mathrm{Fun}(#1,#2)} % the category of functor from #1 to #2
\newcommand{\FcatR}[2]{\mathrm{Fun}^R(#1,#2)} % the category of limit preserving functor from #1 to #2
\newcommand{\FcatRl}[3]{\mathrm{Fun}^{#1R}(#2,#3)}
\newcommand{\FcatRRl}[3]{\mathrm{Fun}^{#3R,#3R}(#1,#2)} % the category of #3-small limits (in each variable) preserving functor from #1 to #2
\newcommand{\Indl}[1]{\mathsf{Ind}_{#1}} % command to denote the completion of a small categories under #1-filtered colimits (#1 a regular cardinal)
\newcommand{\coten}[2]{{#1}^{#2}}
\newcommand{\Xicl}{\widetilde{\Xi}}
\newcommand{\model}[1]{\mathbf{Mod}(#1)}
\newcommand{\hmodel}[1]{\mathbf{hMod}(#1)}
\newcommand{\lder}{\mathbb{L}}
\newcommand{\rder}{\mathbb{R}}
\newcommand{\emap}[1]{\mathrm{Hom}_{#1}}
\newcommand{\segal}{_{\text{Segal}}}
\newcommand{\Alm}[1]{\mathrm{Alg}(#1)}

\newcommand{\colim}{\mathop{\mathrm{colim}}}
\renewcommand{\lim}{\mathop{\mathrm{lim}}}
\newcommand{\hocolim}{\mathop{\mathrm{hocolim}}}
\newcommand{\holim}{\mathop{\mathrm{holim}}}

\numberwithin{equation}{section}

\begin{document}

\begin{abstract}
Given a limit sketch in which the cones have a finite connected base, we show that a model structure of ``up to homotopy'' models for this limit sketch in a suitable model category can be transferred to a Quillen equivalent model structure on the category of strict models. As a corollary of our general result, we obtain a rigidification theorem which asserts in particular that any $\Theta_n$-space in the sense of Rezk is levelwise equivalent to one that satisfies the Segal conditions on the nose. There are similar results for dendroidal spaces and $n$-fold Segal spaces.
\end{abstract}

\begin{asciiabstract}
Given a limit sketch in which the cones have a finite connected base, we show that a model structure of ``up to homotopy'' models for this limit sketch in a suitable model category can be transferred to a Quillen equivalent model structure on the category of strict models. As a corollary of our general result, we obtain a rigidification theorem which asserts in particular that any (Theta_n)-space in the sense of Rezk is levelwise equivalent to one that satisfies the Segal conditions on the nose. There are similar results for dendroidal spaces and n-fold Segal spaces. 
\end{asciiabstract}

\maketitle
%\tableofcontents

%%%%%%%%%%%%%%%%%%%% Start of main body of article

\section{Introduction}

This paper is concerned with the problem of rigidification for higher categorical structures. Usually, the correct definition of a higher categorical object is the one where the relations are required to hold as weakly as possible. However, sometimes, these objects can be partially rigidified to equivalent objects for which some of the relations hold strictly. The simplest example of this phenomenon can be seen on monoidal categories. The standard definition of a monoidal category involves associators and unitors that insure that any two ways of tensoring a finite sequence of objects are uniquely isomorphic. On the other hand, it is known that any monoidal category is equivalent to a monoidal category in which the tensor product is unital and associative on the nose, meaning that all the possible ways of tensoring a finite sequence of objects are equal. More generally, any bicategory is equivalent to a $2$-category. However, this is a lucky accident that does not happen for higher dimensional higher categories. For instance, a tri-groupoid encoding the $3$-type of the $2$-sphere cannot be rigidified to a strict $3$-groupoid (\emph{cf.} \cite[Section 2.7.]{simpsonhomotopy}). These problems become more and more untractable as the dimension gets bigger.

One consequence of our main result can be vaguely phrased by saying that the problem of rigidification of higher categories only comes from the invertible cells and as long as one does not to try to rigidify them, it should be possible to get a strict model of the higher category. A more precise statement of what we prove is that the homotopy theory of $(\infty,n)$-categories is equivalent to that of strict $n$-categories internal to Kan complexes.  According to the homotopy hypothesis, the homotopy theory of any coherent enough notion of $\infty$-groupoid should be equivalent to the homotopy theory of CW-complexes up to homotopy equivalences. It is well-known that the homotopy theory of Kan complexes is equivalent to that of CW-complexes and many mathematicians actually take Kan complexes as a definition of $\infty$-groupoids.

Our method for proving this rigidification result is to put a model structure on the category of $n$-categories internal to simplicial sets that is equipped with a Quillen equivalence to the model category of $\Theta_n$-spaces with the Rezk model structure. It is widely accepted by the mathematical community that $\Theta_n$-spaces form a good model of $(\infty,n)$ categories. In \cite{bergnercomparison}, Bergner and Rezk  compare this model to other reasonable models of $(\infty,n)$-categories. This model also satisfies the axiomatization of $(\infty,n)$-categories given by Barwick and Schommer-Pries in \cite{barwickunicity}.

A partial result in this direction was obtained by the second author in \cite{horelmodel}. One of the main result of that paper was to prove an equivalence between simplicial spaces with the complete Segal space model structure with a model structure on categories internal to simplicial sets. In that case the rigidification result is of course not very surprising because by work of Bergner \cite{bergnerthree}, we even know that a complete Segal space can be rigidified to a simplicially enriched category.

When we tried to generalize the result of \cite{horelmodel} to $(\infty,n)$-categories, we realized that, not only could it be done but also that the proof used very little about the category of $n$-categories. The main point is that $n$-categories form a locally presentable category and that there is a full subcategory $\Theta_n$ such that the associated nerve construction $n\Cat\to\Set^{\Theta_n\op}$ is fully faithful and preserves coproducts and filtered colimits. 

We can thus abstract this situation and consider pairs consisting of a locally presentable category $\cat{X}$ with a full subcategory $\Xi$ satisfying the following assumptions:

\begin{assu}\label{assumptions X}
\hspace{0pt}
\begin{itemize}
 \item $\Xi$ is dense in $\X$, i.e. the functor $N:\X\to \Set^{\Xi\op}$ sending $X$ to $\xi\mapsto \X(\xi,X)$ is fully faithful.
 \item For any $\xi$ in $\Xi$, the functor $\X(\xi,-)$ preserves filtered colimits.
 \item For any $\xi$ in $\Xi$, the functor $\X(\xi,-)$ preserves finite coproducts 
\end{itemize}
\end{assu}

It turns out that the category of $n$-categories is not the only interesting example of such a situation. One can for instance take $\X$ to be the category of colored operads and $\Xi$ to be its subcategory $\Omega$. One can also take $\X$ to be the category $\Cat^{\otimes n}$ of $n$-fold categories and the subcategory $\Delta^n$. In fact, we prove in Proposition \ref{prop: main prop finite connected limit sketch}, that given a locally presentable category $\cat{X}$, the existence of a full subcategory $\Xi$ satisfying the assumptions \ref{assumptions X} is equivalent to the fact that $\cat{X}$ is the category of models for a finite connected limit sketch. 

Our main result is given by the following Theorem.

\begin{theo*}[\ref{theo: main}]
Let $(\cat{X},\Xi)$ be a pair satisfying the assumptions \ref{assumptions X} and let $\S$ be a Grothendieck topos equipped with a combinatorial model structure in which the cofibrations are monomorphisms. The projective model structure on $\S^{\Xi\op}$ as well as any of its left Bousfield localizations can be lifted to $\S\otimes\X$, the category of objects of $\X$ internal to $\S$, via the nerve functor. The resulting adjunction between $\S^{\Xi\op}$ and $\S\otimes\X$ is moreover a Quillen equivalence.
\end{theo*}

We advise the reader to not read this paper linearly but rather to start from the last section where the most interesting applications are developed and to refer to the previous sections as needed. In particular Theorems \ref{theo:main operads}, \ref{theo:main n-cat} and \ref{theo: main n-fold} examine the consequences of the above theorem to the theory of operads, $n$-categories, $n$-fold categories. In each of these cases, this main theorem implies a rigidification result which can be expressed by saying that any homotopy operad (resp. $n$-categories, resp. $n$-fold categories) internal to $\S$ is levelwise equivalent to one which is strict.

\subsection*{Acknowledgements}
\emph{GC :} I wish to thank Dimitri Ara for suggesting to look at this problem, Ieke Moerdijk for his support and Javier J. Guti\'errez, Simon Henry and Joost Nuiten for many insightful discussions.

\emph{GH :} I wish to thank the Hausdorff Institute for Mathematics in Bonn for providing an excellent work environment in the summer 2015 during which most of this paper was written. I also want to thank Claudia Scheimbauer and Chris Schommer-Pries for their interest in this project.

We both thank the anonymous referee for several helpful comments.

\section{A few facts on locally presentable categories}

Let $\lambda$ be a regular cardinal. Given a small category $C$ with all $\lambda$-small colimits, we denote by $\Indl{\lambda}(C)$ the category of functors $C\op\to\Set$ that preserve $\lambda$-small limits. When $\lambda=\omega$ is the smallest infinite cardinal, we write $\Indl{}$ instead of $\Indl{\omega}$.

\begin{defi}
A \textbf{$\lambda$-locally presentable category} is a category $\cat{C}$ that is equivalent to $\Indl{\lambda}(A)$ with $A$ a small category with all $\lambda$-small colimits. A locally presentable category is a category that is $\lambda$-locally presentable for some $\lambda$.
\end{defi}

\begin{prop}\label{prop: criterion lambda-presentability}
Let $\cat{C}$ be a category and $I:\cat{C}\to\Fun(A,\Set)$ be a fully faithful right adjoint that preserves $\lambda$-filtered colimits. Then $\cat{C}$ is $\lambda$-locally presentable.
\end{prop}

\begin{proof}
See \cite[Theorem 1.46]{adamekrosicky}.
\end{proof}

If $\cat{C}$ and $\cat{D}$ are two locally presentable categories, we denote by $\cat{C}\otimes\cat{D}$ their tensor product in the category of locally presentable categories. This is a locally presentable category with a functor $\cat{C}\times\cat{D}\to\cat{C}\otimes\cat{D}$ that preserves colimits in each variable separately and which is initial with this property.

We now recall how this tensor product is constructed explicitly. It is not restrictive to suppose that $\cat{C}$ and $\cat{D}$ are $\lambda$-presentable categories for some regular cardinal $\lambda$.
Then $\cat{C}\simeq  \Indl{\lambda}(A)$ and $\cat{D}\simeq  \Indl{\lambda}(B)$ for some small full subcategory $A$ (resp. $B$) of $\cat{C}$ (resp. $\cat{D}$).
We can then construct the category $\cat{C}\otimes\cat{D}:=\FcatRRl{A\op\times B\op}{\Set}{\lambda}$. This is the full subcategory of $\Fcat{A\op\times B\op}{\Set}$ spanned by the functors that preserve $\lambda$-small limits in each variable. This category is also equivalent to the category $\FcatRRl{\cat{C}\op\times\cat{D}\op}{\Set}{}$ of functors $\cat{C}\op\times\cat{D}\op\to\Set$ preserving limits separately in each variable. Note however that this second definition does not make it obvious that $\cat{C}\otimes\cat{D}$ has small hom-sets.

The functor $\cat{C}\times\cat{D}\to\cat{C}\otimes\cat{D}$ sends $(c,d)\in \cat{C}\times \cat{D}$ to the functor $c\otimes d\in\Fcat{\cat{C}\op\times \cat{D}\op}{\Set}$ such that $c\otimes d(c',d')=\cat{C}(c',c)\times \cat{D}(d',d)$.
for every $c',d'\in \cat{C}\op\times \cat{D}\op$. It is easy to verify that $c\otimes d$ preserves limits in each variables and hence, we have indeed constructed a functor $\cat{C}\times\cat{D}\to\cat{C}\otimes\cat{D}$.

Using the fact that the map $\cat{D}\to\FcatR{\cat{D}\op}{\Set}$ sending $d$ to the limit preserving functor $\cat{D}(-,d)$ from $\cat{D}\op$ to $\Set$ is an equivalence of categories, we see that $\cat{C}\otimes\cat{D}$ is also equivalent to the category $\FcatR{\cat{C}\op}{\cat{D}}$. Note that with this last description, the commutativity of the tensor product is not obvious. Through the equivalence $\cat{C}\otimes\cat{D}\simeq\FcatR{\cat{C}\op}{\cat{D}}$, the object $c\otimes d$ is the limit preserving functor $\cat{C}\op\to\cat{D}$ such that we have a natural isomorphism
\begin{equation}\label{c otimes d}
\FcatR{\cat{C}\op}{\cat{D}}(c\otimes d,F)\cong\cat{D}(d,F(c)).
\end{equation}

Now, we study the functoriality of this tensor product. Let us assume that we are given an adjunction $u^*:\cat{C}\rightleftarrows\cat{D}:u_*$ between locally presentable categories and let $\cat{Z}$ be any locally presentable category. We can construct a functor $u_*\otimes\id:\cat{D}\otimes\cat{Z}\to\cat{C}\otimes\cat{Z}$ by identifying $\cat{C}\otimes\cat{Z}$ with $\FcatR{\cat{C}\op}{\cat{Z}}$ and $\cat{D}\otimes\cat{Z}$ with $\FcatR{\cat{D}\op}{\cat{Z}}$ and taking the functor induced by precomposition with $(u^*)\op$. We can also construct the functor $\id\otimes u_*:\cat{Z}\otimes\cat{D}\to\cat{Z}\otimes\cat{C}$ obtained by identifying $\cat{Z}\otimes\cat{C}$ with $\FcatR{\cat{Z}\op}{\cat{C}}$ and $\cat{Z}\otimes\cat{D}$ with $\FcatR{\cat{Z}\op}{\cat{D}}$ and taking the functor induced by postcomposition with $u_*$.

\begin{lemm}\label{lemm:right-left induced}
Let $u^*\colon \cat{C} \rightleftarrows \cat{D} \colon u_*$ be an adjunction between locally presentable categories and let $\cat{Z}$ be 
another locally presentable category. The diagram
\[
 \xymatrix{\cat{Z}\otimes \cat{C} \ar[d]_{\tau} \ar[r]^{\id\otimes u_*} & \cat{Z}\otimes \cat{D} \ar[d]^{\tau}\\
           \cat{C}\otimes \cat{Z} \ar[r]^{u_*\otimes \id} & \cat{D}\otimes \cat{Z}}
\] 
commutes up to a natural isomorphism. 
\end{lemm}

\begin{proof}
Let $F\in \FcatR{\cat{Z}\op}{\cat{C}}\simeq\cat{Z}\otimes\cat{C} $. Using formula \ref{c otimes d}, we see that the functor $\tau(\id\otimes u_*)(F)$ sends $(d,z)\in \cat{D}\op\times \cat{Z}\op$ to $\cat{D}(d,u_*F(z))$. On the other hand the functor $(u_*\otimes \id)\tau(F)$ sends $(d,z)$ to $\cat{D}(u^*(d),F(z))$. Thus the two functors $(u_*\otimes \id)\tau(F)$ and $\tau(\id\otimes u_*)(F)$ are isomorphic and moreover this isomorphism can be chosen to be functorial in $F$. This proves the commutativity of the square.
\end{proof}

Keeping the same notations as before, we can also construct a functor $\cat{Z}\times\cat{C}\to\cat{Z}\otimes\cat{D}$ sending $(z,c)$ to $z\otimes u^*(c)$. This functor preserves colimits in both variables and hence by the universal property of the tensor product induces a colimit preserving functor $\cat{Z}\otimes\cat{C}\to\cat{Z}\otimes\cat{D}$ that we denote by $\id\otimes u^*$. 

\begin{lemm}\label{lemm:adjunction tensor product}
The functor $\id\otimes u_*$ is right adjoint to $\id\otimes u^*$.
\end{lemm}

\begin{proof}
Let $F:\cat{Z}\op\to\cat{D}$ be an object of $\cat{Z}\otimes\cat{D}\simeq\FcatR{\cat{Z}\op}{\cat{D}}$. Using equation \ref{c otimes d}, we find a sequence of natural isomorphisms
\[\cat{Z}\otimes\cat{D}((\id\otimes u^*)(z\otimes c),F)\cong\cat{Z}\otimes\cat{D}(z\otimes u^*c,F)\cong\cat{Z}(z,F(u^*c))\cong\cat{Z}\otimes\cat{C}(z\otimes c,(\id\otimes u_*)F)\]
which proves the desired result.
\end{proof}

\begin{coro}\label{coro:induced functor}
Let $\cat{C}$ and $\cat{D}$ be locally presentable categories and let $u^*\colon \cat{C} \rightleftarrows \cat{D} \colon u_*$ be an adjunction
between them. Suppose that $u^*$ commutes with $\lambda$-small limits. Let $\cat{Z}\simeq\Indl{\lambda}(A)$ be a $\lambda$-locally presentable category. Then, the functor $\id\otimes u^*$ is isomorphic to the composite
\[\cat{Z}\otimes\cat{C}\simeq\FcatRl{\lambda}{A\op}{\cat{C}}\xrightarrow{ u^*\circ -}\FcatRl{\lambda}{A\op}{\cat{D}}\simeq\cat{Z}\otimes\cat{D}\]
where the middle map is given by postcomposition with $u^*$
\end{coro}

\begin{proof}
By lemma \ref{lemm:adjunction tensor product} and the unicity of a left adjoint, it suffices to prove that postcomposition with $u^*$ is left adjoint to $\id\otimes u_*$. But since $\id\otimes u_*$ is isomorphic to the composite
\[\cat{Z}\otimes\cat{D}\simeq\FcatRl{\lambda}{A\op}{\cat{D}}\xrightarrow{ u_*\circ -}\FcatRl{\lambda}{A\op}{\cat{C}}\simeq\cat{Z}\otimes\cat{C},\]
the result is obvious.
\end{proof}

\section{The setup}

Let $\X$ be a category with all colimits and let $\Xi$ be a small full subcategory of $\X$. We assume that the pair $(\X,\Xi)$ satisfies the assumptions \ref{assumptions X}. Let us denote by 
$S:~\cat{Set}^{\Xi\op}\to~\X$ the left Kan extension of
the inclusion $\Xi \rightarrow \X$ along the Yoneda embedding $y:~\Xi\to~\cat{Set}^{\Xi\op}$. The functor $S$ has always a right adjoint, the \textbf{nerve} functor, that we denote by $N$. Concretely, if $A$ is an object of $\cat{X}$, then $NA(\xi)=\cat{X}(\xi,A)$.

Using the first point of \ref{assumptions X}, we see that $\X$ is a full reflective subcategory of $\cat{Set}^{\Xi\op}$ via the adjunction $(S,N)$. In particular, according to Proposition \ref{prop: criterion lambda-presentability}, $\X$ is an $\omega$-locally presentable category and hence has all small limits. The limits in $\cat{X}$ are created by the nerve functor. Under the other two points of assumptions \ref{assumptions X}, we find that the finite coproducts and filtered colimits in $\X$ are also created by the nerve functor.

\begin{example}
The main examples that we have in mind are the following. They will be developed in details in the last section.
\begin{enumerate}
\item The category $\X$ is the category $\cat{Cat}$ of small categories. The category $\Xi$ is the category $\Delta$ of finite linearly ordered sets. The functor $N$ is the usual nerve functor from $\cat{Cat}$ to simplicial sets.
\item The category $\X$ is the category $\cat{Op}$ of small colored operads in sets. The category $\Xi$ is the category $\Omega$ of dendrices. The functor $N$ is the dendroidal nerve functor from colored operads to dendroidal sets.
\item The category $\X$ is the category of $n$-fold categories, $\Xi$ is $\Delta^n$.
\item The category $\X$ is the category of $n$-categories, $\Xi$ is Joyal's theta category $\Theta_n$.
\item $\X$ is the category of models for a finite connected limit sketch. In that case the full subcategory $\Xi$ can be constructed as explained in Section \ref{sec:finite connected limit sketches}.
\end{enumerate}
\end{example}

We also have a category of geometric objects $\S$ that we generically call ``spaces''. We make the following assumptions on $\S$.

\begin{assu}\label{assumption S}
We assume that $\S$ is a Grothendieck topos equipped with a combinatorial model structure in which the cofibrations are monomorphisms. 
\end{assu}

\begin{example}
The prototypical example of a category $\S$ satisfying the assumptions of \ref{assumption S} is the category $\sSet$ of simplicial sets with its standard model structure. If $I$ is a small category, the category of simplicial presheaves on $I$, $\Set^{\Delta\op\times I\op}$ with its projective or injective model structure also satisfies our hypothesis. If $I$ is a site then the Joyal model structure (\emph{cf.} \cite[II.Theorem 5.9]{jardine2015local}) on the category of simplicial sheaves on $I$ denoted $\cat{Sh}(I,\sSet)$ also satisfies the hypothesis. 

It is also the case that if $\S$ is a category equipped with a model structure satisfying \ref{assumption S}, the same is true for any Bousfield localization of $\S$. Thus, for any site $I$, the category $(\sSet)^I$ with the local model structure of Jardine's (\cite[II.Theorem 5.8]{jardine2015local}) also satisfies the assumptions \ref{assumption S}. Another example of interest that satisfies \ref{assumption S} is the model category of motivic spaces. It is obtained by localizing the Jardine model structure of presheaves over the Nisnevich site of a field $k$ at the $\mathbb{A}^1$-equivalences.

Note however that the usual model structure on topological spaces does not satisfy the assumptions \ref{assumption S}. The main issue is that topological spaces do not form a topos.
\end{example}

\begin{prop}\label{prop: N pi=pi N}
Let $(\X,\Xi)$ be a pair satisfying assumptions \ref{assumptions X}. Let $\pi \colon \cat{R} \to \S $ be a geometric morphism between Grothendieck toposes. The diagram
\[
\xymatrix{
\X\otimes\S\ar[r]^-{N\otimes\id}\ar[d]_{\id\otimes\pi^*}&\S^{\Xi\op}\ar[d]^{\id\otimes\pi^*}\\
\X\otimes\cat{R}\ar[r]_-{N\otimes \id}&\cat{R}^{\Xi\op}
}
\]
commutes up to isomorphism.
\end{prop}

\begin{proof}
We indentify $\X\otimes \S$ with $\FcatR{\X\op}{\S}$.
By Lemma \ref{lemm:right-left induced} both $N\otimes\id$'s in the diagram are given by pre-composition with $S$ and by Corollary \ref{coro:induced functor}
both $\id\otimes\pi^*$'s in the diagram are given by post-composition with $\pi^*$.  
\end{proof}

\begin{rem}\label{rem:monomorphisms}
Recall that a geometric morphism $\pi \colon \cat{R} \to \S $ between toposes is \textbf{surjective} if and only if the inverse image functor $\pi^*$ is conservative. In a topos, a map $X\to Y$ is a monomorphism if and only if the induced map $X\to X\times_YX$ is an isomorphism. Given a surjective geometric morphism $\pi\colon\cat{R}\to\S$, the functor $\pi^*:\S\to \cat{R}$ preserves finite limits and is conservative, therefore it preserves and reflects monomorphisms. 
\end{rem}

\begin{prop}\label{prop:reflective localization}
Let $(\X,\Xi)$ be a pair of categories satisfying the assumptions \ref{assumptions X} and $\S$ a Grothendieck topos. The functor $\id\otimes N:\S\otimes\X\to \S^{\Xi\op}$ exhibits $\S\otimes\X$ as a reflective subcategory of $\S^{\Xi\op}$.
\end{prop}

\begin{proof}
This functor has a left adjoint given by $S:=\id\otimes S:\S\otimes\Set^{\Xi\op}\to \S\otimes\X$. Moreover $\id\otimes N$ is clearly fully faithful.
\end{proof}

\begin{lemm}\label{lemm:pi^*creates colimts}
Let $\pi \colon \cat{R} \to \S $ be a surjective geometric morphism between Grothendieck toposes. The functor $\pi^*\otimes\id:\S\otimes\X\to\cat{R}\otimes \X$ creates colimits.
\end{lemm}

\begin{proof}
A functor creates colimits if it is conservative and preserves colimits. The functor $\pi^*\otimes\id$ is a left adjoint and therefore preserves colimits. We claim that it is also conservative. Let $i$ be a map in $\S\otimes\X$ such that $\pi^*\otimes\id(i)$ is an isomorphism. Then $(\id\otimes N)\circ(\pi^*\otimes\id)(i)$ is an isomorphism. According to Proposition \ref{prop: N pi=pi N}, we see that $(\pi^*\otimes\id)\circ(\id\otimes N)(i)$ is an isomorphism. By Proposition \ref{prop:reflective localization}, the functor $\id\otimes N$ is conservative. Hence, it is enough to prove the conservativity of $\pi^*\otimes\id:\S^{\Xi\op}\to\cat{R}^{\Xi\op}$ which follows immediately from the conservativity of $\pi^*$.
\end{proof}

\section{Finite connected limit sketches}\label{sec:finite connected limit sketches}

\begin{defi} (\cite[1.49]{adamekrosicky})
A \textbf{limit sketch} is a couple $(\cat{T},L)$ where $\cat{T}$ is a small category and $L$ is a collection of cones in $\cat{T}$.
\end{defi}

Given a complete category $\cat{Y}$, a $\cat{Y}$-model for $(\cat{T},L)$ is a functor $F\colon \cat{T}\to \cat{Y}$ sending all cones in $L$ to limit cones in $\cat{Y}$. The full subcategory of $\cat{Y}^{\cat{T}}$ spanned by the $\cat{Y}$-models for $(\cat{T},L)$ will be denoted by $\model{\cat{T},L}_{\cat{Y}}$. The category $\model{\cat{T},L}_{\Set}$ is simply denoted $\model{\cat{T},L}$. It is a full reflective subcategory of $\Set^{\cat{T}}$  and it is therefore locally presentable. Conversely, it can be shown that any locally presentable category is equivalent to the category of models of a limit sketch (see \cite[Corollary 1.52]{adamekrosicky}).

For every locally presentable category $\cat{Y}$ the tensor product $\model{\cat{T},L}\otimes \cat{Y}$ is equivalent to $\model{\cat{T},L}_\cat{Y}$. In particular, for any locally presentable category $\cat{Y}$, the category $\model{\cat{T},L}_\cat{Y}$ is locally presentable.

\begin{defi}
A \textbf{finite connected limit sketch} is a sketch $(\cat{T},L)$ such that all cones in $L$ are indexed by finite connected diagrams. 
\end{defi}

The category $\model{\cat{T},L}$ of $\Set$-models for a finite connected sketch $(\cat{T},L)$ is locally finitely presentable. Moreover, it is closed under coproducts and filtered colimits in $\Set^{\cat{T}}$. In other words, the inclusion functor $i$ in the adjunction
\[
l\colon \Set^\cat{T}\rightleftarrows \model{\cat{T},L} \colon i
\] 
preserves coproducts and filtered colimits. More generally we have the following Lemma.

\begin{lemm}\label{lemm:connected limit topos}
Let $\S$ be a Grothendieck topos. In the adjunction
\begin{equation}\label{eq:Ssketch}
l\otimes \id \colon \Set^{\cat{T}}\otimes \S \rightleftarrows \model{\cat{T},L}\otimes \S \colon i\otimes \id,
\end{equation}
the right adjoint $i\otimes \id$ preserves coproducts and filtered colimits. 
\end{lemm}

\begin{proof}
It is sufficient to prove that $\model{\cat{T},L}\otimes\S$ is closed under coproducts and filtered colimits as a subcategory of $\S^{\cat{T}}$.
This follows from the fact that in any Grothendieck topos finite limits commute with filtered colimits and connected limits commute with coproducts. 
\end{proof}

The rest of this section will be devoted to the proof of the following proposition.

\begin{prop}\label{prop: main prop finite connected limit sketch}
Let $\X$ be a locally presentable category. The following are equivalent.
\begin{enumerate}
\item There exists a small full subcategory $\Xi$ of $\X$ such that the pair $(\X,\Xi)$ satisfies the assumptions \ref{assumptions X}.
\item There exists a finite connected limit sketch $(\cat{T},L)$ and an equivalence of categories $\cat{X}\simeq \model{\cat{T},L}$.
\end{enumerate}
\end{prop}

We first prove that (1) implies (2) in Proposition \ref{prop: main prop finite connected limit sketch}. We consider a cocomplete category $\X$ with a dense full subcategory $\Xi$ of compact and connected objects. We define $\Xicl$ to be the closure of $\Xi$ under finite connected colimits in $\X$. Then $\Xicl$ is dense and spanned by connected and compact objects; thus the nerve functor $\widetilde{N}\colon \X \to \Set^{\Xicl\op}$ is fully faithful. Let $L$ be the set of representatives of all the finite connected limit cones in $\Xicl\op$, then we have the following result.

\begin{lemm}
The essential image of $\widetilde{N}$ is the category of $(\Xicl\op,L)$-models.
\end{lemm} 

\begin{proof}
First notice that for every $X\in \X$ the functor $\widetilde{N}(X)$ sends all cones in $L$ to limit cones; indeed if $D\colon I \to \Xicl$ is a finite connected diagram in $\Xicl$ then
\[\widetilde{N}(X)(\colim_{i\in I} D(i))\cong \X(\colim_{i\in I} D(i),X)\cong \lim_{i\in I} \X(D(i),X)\cong \lim_{i\in I} \widetilde{N}(X)(D(i)).\]
Thus it is sufficient to prove that every functor $F \colon \Xicl\op \to \Set$ that preserves finite connected limits is in the essential image of $\widetilde{N}$. The essential image of $\widetilde{N}$ is closed under filtered colimits and coproducts; since the representables are all contained in the image of $\widetilde{N}$, it is sufficient to show that $F$ is a coproduct of filtered colimits of representables. Since the category of representables is dense in $\Set^{\Xicl\op}$, it is sufficient to show that the connected components of the comma category $\Xicl \downarrow F$ are filtered.

This amounts to showing that, for every finite connected category $I$, every diagram $D\colon I \to \Xicl \downarrow F$ is the base of a cocone in $\Xicl \downarrow F$. Let $\pi\colon \Xicl \downarrow F \to \Xicl$ be the canonical projection. By Yoneda lemma, giving the diagram $D$ is equivalent to giving a system of elements $\{x_i\in F(\pi D(i))\}_{i\in I}$ such that $F(f)(x_j)=x_i$ for every $f\colon i \to j$ in $I$. In other words, the data of the diagram $D$ is exactly the data of an element \[(x_i)_{i\in I} \in \underset{i\in I}\lim F(\pi D(i)).\] Let $d\in \Xicl$ be the colimit of $\pi D$. By assumption 
\[F(d)\cong \underset{i\in I}\lim F(\pi D(i));\] 
the $(x_i)_{i\in I}$ determines an arrow $d \to F$ in $\Set^{\Xicl\op}$ which, seen as an object of  $\Xicl \downarrow F$, is a vertex for a cocone over $D$.   		
\end{proof}

\begin{coro}\label{coro:coproducts pres topos}
For every Grothendieck topos $\S$, the nerve functor $N\otimes \id\colon \X\otimes \S \to \S^{\Xi\op}$ preserves filtered colimits and coproducts.
\end{coro}

\begin{proof}
The functor $N\otimes \id$ is isomorphic to the composite
\[\X\otimes \S 
 	          \xrightarrow{\widetilde{N}\otimes \id}\S^{\Xicl\op}
 	          \xrightarrow{l^*}\S^{\Xi\op},   
\]
where $l \colon \Xi\op \to \Xicl\op$ is the canonical inclusion. The functor $\widetilde{N}\otimes \id$ preserves coproducts and filtered colimits by Lemma \ref{lemm:connected limit topos}. The functor $l^*$ is left adjoint, hence it preserves all colimits.
\end{proof}

Now, we want to prove that (2) implies (1) in Proposition \ref{prop: main prop finite connected limit sketch}. We start with a definition.

\begin{defi}\label{defi: fully faithful}
Let $(\cat{T},L)$ be a limit sketch, we say that it is \textbf{fully faithful} if for any $t\in \cat{T}$, the functor $\cat{T}(t,-)$ is a model for $(\cat{T},L)$.
\end{defi}

Let $(\cat{T},L)$ be a finite connected limit sketch and let $\cat{X}$ be the category of its $\Set$-models. The inclusion $i\colon \cat{X} \to \Set^{\cat{T}}$ has a left adjoint $l\colon \Set^{\cat{T}} \to \cat{X}$. 

\begin{prop}\label{prop: fully faithful limit sketch}
Assume that $(\cat{T},L)$ is fully faithful. Then the category $\cat{T}\op$ seen as a full subcategory of $\cat{X}$ is a dense subcategory whose objects are compact and connected.
\end{prop}

\begin{proof}
The density of the inclusion $\cat{T}\op\to\cat{X}$ is obvious. Let $t$ in $\cat{T}$, we want to show that $h_t=\cat{T}(t,-)$ is connected and compact in $\cat{X}$. Let $F:D\to\cat{X}$ be a diagram, then we have
\[\cat{X}(h_t,\colim_{d\in D}F(d))\cong\Set^{\cat{T}}(h_t,i(\colim_{d\in D}F(d))).\]

Hence, if $D$ is a filtered category or a discrete category, we have
\[\cat{X}(h_t,\colim_{d\in D}F(d))\cong\Set^{\cat{T}}(h_t,\colim_{d\in D}i(F(d)))\cong \colim_{d\in D}\Set^{\cat{T}}(h_t,i(F(d))),\]
where the second equality follows from the fact that colimits are computed objectwise in a presheaf category. It follows that the object $h_t\in \cat{X}$ is compact and connected.
\end{proof}

Now, let $(\cat{T},L)$ be a general finite connected limit sketch. Let $\cat{X}$ be the category of its $\Set$-models. Let $\cat{A}=l\circ y(\cat{T}\op)$ be the image of $l$ composed with the Yoneda embedding. The category $\cat{A}$ is dense in $\cat{X}$, moreover, the objects of $\cat{A}$ are compact and connected in $\X$ since $i$ preserves filtered colimits and coproducts. Let $\theta\colon \cat{T} \to \cat{A}\op$ be the map induced by restricting $(l\circ y)\op$. There is an induced adjunction
\[
\xymatrix{\Set^{\cat{T}}\ar@<-3pt>[r]_-{\theta_*} & \ar@<-3pt>[l]_-{\theta^*} \Set^{\cat{A}\op}.} 
\]
We can then consider the finite connected limit sketch $(\cat{A}\op,\theta(L))$. The following proposition shows that it is fully faithful and that its category of models is equivalent to $\cat{X}$ which according to Proposition \ref{prop: fully faithful limit sketch} will conclude the proof of Proposition \ref{prop: main prop finite connected limit sketch}.

\begin{prop}\label{prop:full subcategory sketch} 
The adjunction $(\theta^*,
\theta_*)$ restricts to an equivalence of categories
\[
\xymatrix{\model{\cat{T},L}_{\Set}\ar@<-3pt>[r]_-{\theta_*} & \ar@<-3pt>[l]_-{\theta^*} \model{\cat{A}\op,\theta(L)}_{\Set}}.
\]	
\end{prop}

\begin{proof}
For every $t\in \cat{T}$, we denote the representable functor $\cat{T}(t,-)$ by $h_t$. For every $a\in \cat{A}$ the functor $i(a)\in \Set^{\cat{T}}$ is isomorphic to 
\[\underset{\{t\to i(a)\}\in \cat{T}\op\downarrow i(a)}\colim \hspace{-15pt} h_t,\] 
and, by adjunction, the indexing category $\cat{T}\op\downarrow i(a)$ is isomorphic to $\theta\op\downarrow a$.
 
Suppose that $G\in \Set^{\cat{T}}$ belongs to $\model{\cat{T},L}_{\Set}$. 
For every $a\in \cat{A}\op$, $\theta_*(G)$, the right Kan extension of $G$ along $\theta$, has $a$-component 
\[\theta_*G(a)\cong \underset{\{a\to \theta(t)\}\in a\downarrow \theta}\lim\hspace{-5pt}  G(t)\cong \Set^{\cat{T}} ( \underset{\{\theta(t)\to a\}\in \theta\op\downarrow a} \colim\hspace{-7pt} h_t,G)\cong \Set^{\cat{T}}(i(a),G)\cong \model{\cat{T},L}(a,G).\]
It follows that for every $t\in \cat{T}$, we have
\[
 \theta^*\theta_*G(t)\cong \model{\cat{T},L}(l(t),G)\cong \Set^{\cat{T}}(y(t),G)\cong G(t).
\]
This implies that the counit $\varepsilon \colon \theta^*\theta_*(G) \to G$ is an isomorphism. Since, by assumption, $G$ belongs to $\model{\cat{T},L}_{\Set}$, this also implies that $\theta_*G$ belongs to $\model{\cat{A},\theta(L)}_{\Set}$.

Note that for every $a\in \cat{A}$ the representable functor $h_a\in \Set^{\cat{A}\op}$ is a model for $(\cat{A},\theta(L))$. 
Conversely if $F\in \Set^{\cat{A}\op}$ is a model for $(\cat{A},\theta(L))$, then $\theta^*F$ is a model for  $(\cat{T},L)$; furthermore the unit $\eta \colon F \to \theta_*\theta^*(F)$ is an isomorphism in fact, for every $a\in \cat{A}\op$
\[
 F(a)\cong \Set^{\cat{T}}(i(a),\theta^*(F))\cong \Set^{\cat{T}}(\underset{\{\theta(t)\to a\}\in \theta\op \downarrow a}\colim  h_t,\theta^*F)\cong \theta_*\theta^*F(a).
\]
\end{proof}

\section{Construction of the model structure}

In all this section, we fix a pair $(\X,\Xi)$ satisfying assumptions \ref{assumptions X}. For $\cat{Y}$ and $\cat{Z}$ two locally presentable categories, we denote by $(Y,Z)\mapsto Y\otimes Z$ the canonical functor $\cat{Y}\times \cat{Z}\to \cat{Y}\otimes \cat{Z}$. 
  
\begin{lemm}\label{lemma:key lemma}
Let $\S$ be a Grothendieck topos. Let $C$ be an object in $\X$ and let $i:K \to L$ is a monomorphism in $\S$. If 
\[\xymatrix{
K\otimes C\ar[d]\ar[r]&X\ar[d]\\
L\otimes C\ar[r]&Y
}
\]
is a pushout diagram in $\S\otimes\X$, then its nerve is a pushout diagram in $\S^{\Xi\op}\cong \S\otimes\Set^{\Xi\op}$.
\end{lemm}

\begin{proof}
Let us denote by $\square$ the above square. By \cite[Theorem 7.54] {johnstonetopos} there exists $\widetilde{\S}$ a Boolean topos with a surjective geometric morphism $\pi\colon \widetilde{\S} \to \S$. Recall from Proposition \ref{prop: N pi=pi N} that the square
\[
\xymatrix{
\S\otimes\X\ar[d]_{\pi^*\otimes\id}\ar[r]^-{\id\otimes N}&\S\otimes\Set^{\Xi\op}\ar[d]^{\pi^*\otimes\id}\\
\widetilde{\S}\otimes\X\ar[r]_-{\id\otimes N}&\widetilde{\S}^{\Xi\op}
}
\]
commutes up to natural isomorphism. Moreover, the vertical maps create colimits. 

The square $\square$ lives in the upper left corner. We want to prove that $\id\otimes N(\square)$ is a pushout square. By \ref{lemm:pi^*creates colimts}, we know that $\pi^*\otimes\id$ creates colimits. Hence $\id\otimes N(\square)$ is a pushout if and only if $(\pi^*\otimes\id)\circ(\id\otimes N)(\square)$ is a pushout square. By commutativity of the above diagram,  $(\id\otimes N)(\square)$ is a pushout if and only if $(\id\otimes N)\circ(\pi^*\otimes\id)(\square)$ is a pushout. We know by assumption that $\square$ is a pushout, therefore, since $\pi^*\otimes\id$ is a left adjoint, $(\pi^*\otimes\id)(\square)$ is a pushout. Hence, we want to prove that $(\id\otimes N)(\square)$ is a pushout knowing that $(\id\otimes N)(\pi^*\otimes\id)(\square)$ is a pushout. In other words, we are reduced to prove the lemma in the case where $\S=\widetilde{\S}$ is a Boolean topos.

Using Remark \ref{rem:monomorphisms}, we are reduced to proving that if $K\to L$ is a monomorphism in $\widetilde{\S}$, and the square
\[\xymatrix{
K\otimes C\ar[r]\ar[d]&X\ar[d]\\
L\otimes C\ar[r]&Y
}
\]
is  a pushout square in $\widetilde{\S}\otimes \X$, then the square
\[\xymatrix{
K\otimes N(C)\ar[r]\ar[d]&N(X)\ar[d]\\
L\otimes N(C)\ar[r]&N(Y)
}
\]
is  a pushout square in $\widetilde{\S}^{\Xi\op}$. Since $\widetilde{\S}$ is Boolean, we can write $L=K\sqcup (L-K)$ where $L-K$ denotes a complement of $K$. Hence, we know that $Y=(L-K)\otimes C\sqcup X$. Since $N$ commutes with coproducts (Corollary \ref{coro:coproducts pres topos}), we have
\[N(Y)=N(X)\sqcup N((L-K)\otimes C)\cong L\otimes N(C)\sqcup^{K\otimes N(C)} N(X)\]
as desired.
\end{proof}

Now, $\S$ denotes a Grothendieck topos possessing a model structure satisfying the assumptions \ref{assumption S}. We let $I$ (resp. $J$) be a set of generating (trivial) cofibrations for $\S$. The category $\S^{\Xi\op}$ can then be equipped with a projective model structure in which the weak equivalences (resp. fibrations) are the maps that are objectwise weak equivalences (resp. fibrations) in $\S$. We denote by $y:\Xi\to\Set^{\Xi\op}$ the Yoneda embedding. Then a set of generating cofibrations (resp. trivial cofibrations) for the projective model structure on $\S^{\Xi\op}\cong \S\otimes \Set^{\Xi\op}$ is
\[
\mathscr{I}=\{ i\otimes y(\xi) \mid i\in I, \ \xi \in \Xi \} \ \text{(resp. } \mathscr{J}=\{ j\otimes y(\xi) \mid j\in J, \xi \in \Xi \} \text{ ).}
\]

We then have the following proposition.

\begin{prop}\label{cor.key}
The functor $\id \otimes N:\S\otimes \X\to \S^{\Xi\op}$ sends pushouts of maps in $(\id\otimes S)(\mathscr{I})$ (resp. $(\id\otimes S)(\mathscr{J})$) to $\mathscr{I}$-cofibrations (resp. $\mathscr{J}$-cofibrations). 
\end{prop}

\begin{proof}
Let $i\otimes y(\xi)\in \mathscr{I}$ (resp. $i\otimes y(\xi) \in \mathscr{J}$). Then $(\id\otimes S)(i\otimes y(\xi))\cong i\otimes \xi$.
By Lemma \ref{lemma:key lemma} $\id \otimes N$ preserves pushout along such maps; since $(\id\otimes N)(i\otimes \xi)\cong i\otimes y(\xi)\in \mathscr{I}$ (resp. $\mathscr{J}$)
the statement follows.   
\end{proof}

\begin{prop}\label{prop: projective model structure}
There is a model structure on $\S\otimes \X$ whose weak equivalences and fibrations are created by the nerve functor $(\id\otimes N)\colon \S\otimes \X\to \S^{\Xi\op}$. The adjunction $(\id\otimes S,\id \otimes N)$ becomes a Quillen adjunction with respect to this model structure. Moreover the nerve functor preserves cofibrations.
\end{prop}

\begin{proof}
We transfer the projective model structure on $\S^{\Xi\op}$ to
$\S\otimes \X$ using the adjunction $(\id\otimes S,\id \otimes N)$. 
Since all objects of $\Xi$ are compact $\id \otimes N$ preserves filtered colimits.
The statement is an application of \cite[Theorem 1.2]{horelmodel} using Proposition \ref{cor.key}.
\end{proof}

\begin{prop}\label{prop: left proper}
Suppose that the model structure on $\S$ is left (resp. right) proper. Then the model structure on $\S\otimes \X$ of Proposition \ref{prop: projective model structure} is left (resp. right) proper.	
\end{prop}

\begin{proof}
Since $(\id\otimes N)$ preserves pullbacks, fibrations and reflects weak equivalences, the fact that $\S\otimes \X$ is right proper if $\S$ is right proper follows directly from the fact that the projective model structure on $\S^{\Xi\op}$ is right proper.
	
Suppose now that $\S$ is left proper. In order to prove that $\S\otimes \X$ is left proper, we have to prove that for every cofibration $s\colon A \to B$ and every weak equivalence $w\colon A \to X$ the map $w'$ in the pushout square
	\begin{equation}\label{eq: left proper}
	 \xymatrix{
	 	A 
	 	\ar[r]^-{w}
	 	\ar[d]_-{s} &
	 	X
	 	\ar[d] 
	 	\\
	 	B
	 	\ar[r]_-{w'} &
	 	Y		               
	 }
	\end{equation}
is again a weak equivalence. Since every cofibration is the retract of a relative $(\id \otimes S)(\mathscr{I})$-cell complex it is sufficient to consider the case in which $s$ is a relative $(\id \otimes S)(\mathscr{I})$-cell complex.	Since $(\id\otimes N)$ reflects weak equivalences it is sufficient to check that $(\id\otimes N)(w')$ is a weak equivalence.
	
Now, we prove that $(\id\otimes N)$ sends pushouts along relative $(\id \otimes S)(\mathscr{I})$-cell complexes to pushout squares. The statement will follow by applying $(\id\otimes N)$ to (\ref{eq: left proper}), since $(\id\otimes N)$ preserves cofibrations by Proposition \ref{prop: projective model structure} and the projective model structure on $\S^{\Xi\op}$ is left proper.
	
We are left to prove that $(\id\otimes N)$ sends pushouts along a relative $(\id \otimes S)(\mathscr{I})$-cell complex $s$ to pushout squares. Note that $(\id\otimes N)$ preserves filtered colimits, thus it is sufficent to prove our claim in the case in which $s$ is the pushout of a map $i\otimes \xi \in (\id \otimes S)(\mathscr{I})$ for some $i\in I$, $\xi\in \Xi$. In this last case, we have a commutative diagram of pushout squares
	\[
	\xymatrix{K\otimes \xi 
		        \ar[r] 
		        \ar[d]_{i\otimes \xi} & 
		      A 
		        \ar[r]^-{}
		        \ar[d]_-{s} &
		      X
		        \ar[d] 
		      \\
		      L\otimes \xi
		        \ar[r] &
		      B
		       \ar[r]_-{} &
		      Y		               
		     },
	\]
and we have to prove that $(\id\otimes N)$ sends the right square to a pushout square. By Lemma \ref{lemma:key lemma} the left square and the external square are both sent to pushout squares by $(\id\otimes N)$; it follows that the right square is sent to a pushout square as well. 	    
\end{proof}
	
Now, we prove that the Quillen pair $(\id \otimes S,\id \otimes N)$ of Proposition \ref{prop: projective model structure} is a Quillen equivalence.
This will follow from the next proposition.

\begin{prop}\label{prop: iso unit}
Let $f:X\to Y$ be a cofibration in $\S^{\Xi\op}$. Assume that the unit map $X\to (\id \otimes N)(\id \otimes S)(X)$ is an isomorphism. Then the unit map $Y\to (\id \otimes N)(\id \otimes S)(Y)$ is an isomorphism.
\end{prop}

\begin{proof}
Since isomorphisms are closed under retracts and cofibrations in $\S^{\Xi\op}$ are retracts of $\mathscr{I}$-cells complexes, we can assume, without loss of generality, that $f$ is an $\mathscr{I}$-cell complex. Suppose thus that $f$ is the transfinite composition indexed by an ordinal $\beta$
\[
X\cong X_0 \to X_1 \to \dots \to \colim_{\alpha<\beta} X_\alpha\cong Y=:X_{\beta}
\]
in which every map is a pushout along a map of the kind $i\otimes \xi$ for some $i\in I$ and $\xi\in \Xi$.   
Now, we prove by transfinite induction that for each $\alpha$ the unit map $X_\alpha\to (\id \otimes N)(\id \otimes S)(X_\alpha)$ is an isomorphism. We already know that the unit map $X\to (\id \otimes N)(\id \otimes S)(X)$ is an isomorphism. Assume that the result holds for $X_\alpha$. We have the following diagram
\[
\xymatrix{
%0-th row
K\otimes \xi %(0,0)
  \ar[d]_-{i\otimes \xi}
  \ar[r] &
X_\alpha %(0,1)
  \ar[d]
  \ar[r]^-{\cong} &
(\id\otimes N)(\id\otimes S) (X_{\alpha}) %(0,2)
  \ar[d]
\\%1-th row
L\otimes \xi %(1,0)
  \ar[r] &
X_{\alpha+1} %(1,1)
  \ar[r] &
(\id\otimes N)(\id\otimes S) (X_{\alpha+1}) %(1,2)
}
\]
in which the left square is a pushout square, the external square is also a pushout square by Lemma \ref{lemma:key lemma}; it follows that
the right square is also a pushout square and therefore the unit $X_{\alpha+1} \to (\id\otimes N)(\id\otimes S) (X_{\alpha+1})$
is an isomorphism.

Finally, if $\alpha\leq \beta$ is a limit ordinal and the result holds for $X_{\alpha'}$ for any $\alpha'<\alpha$ then it holds for $X_\alpha\cong \colim_{\alpha'<\alpha}X_{\alpha'}$ as well
since a filtered colimit of isomorphisms is an isomorphism.
\end{proof}

\begin{coro}\label{coro:cofibrant objects}
Let $X$ be a cofibrant object of $\S^{\Xi\op}$, then the unit map $X\to (\id\otimes N)\circ(\id\otimes S)(X)$ is an isomorphism.
\end{coro}

\begin{proof}
This is just Proposition \ref{prop: iso unit} applied to the map $\varnothing\to X$, observing that both $(\id \otimes S)$ and $(\id \otimes N)$ preserve the
initial object (this uses the fact that the objects of $\Xi$ are connected in $\X$).
\end{proof}

\begin{prop}\label{prop: Quillen equivalence projective}
The Quillen adjunction of Proposition \ref{prop: projective model structure}
\[
(\id\otimes S)\colon \S^{\Xi\op}\leftrightarrows \S\otimes \X \colon (\id \otimes N)
\]
is a Quillen equivalence.
\end{prop}

\begin{proof}
Let $f\colon (\id\otimes S)(X) \to Y$ be a map in $\S\otimes \X$ where $X$ is a cofibrant object in $\S^{\Xi\op}$ and $Y$ is a fibrant object in $\S \otimes \X$ ; we have to prove that $f$ is a weak equivalence if and only if composition of $(\id\otimes N)(f)$ with the unit map of the adjunction $X\to (\id\otimes N)(\id\otimes S)(X)\to (\id\otimes N)(Y)$ is a weak equivalence. Since the unit map is an isomorphism by Corollary \ref{coro:cofibrant objects} and the functor $(\id\otimes N)$ reflects and preserves weak equivalences, the statement follows from the two-out-of-three property for weak
equivalences.   
\end{proof}

We have also the following characterization of the cofibrations.

\begin{prop}\label{prop:characterization of cofibrations}
A map $f:A\to B$ in $\S\otimes\X$ is a cofibration if and only if $(\id\otimes N)(f)$ is a projective cofibration.
\end{prop}

\begin{proof}
We already know that $(\id\otimes N)$ preserves cofibrations.  Conversely, if $(\id\otimes N)(f)$ is a cofibration, then $(\id\otimes S)\circ(\id\otimes N)(f)$ is a cofibration. But since the functor $(\id\otimes S)\circ(\id\otimes N)$ is isomorphic to the identity functor, we are done.
\end{proof}

We can now state and prove our main theorem.

\begin{theo}\label{theo: main}
Assume that $\S$ is left proper and that $\S^{\Xi\op}$ is equipped with a model structure that is a left Bousfield localization of the projective model structure with respect to a set of maps. Then there is a model structure on $\S\otimes\X$ in which a map is a fibration (resp. cofibration, weak equivalence) if and only if its image by $\id\otimes N$ is a fibration (resp. cofibration, weak equivalence). The functor $\id\otimes N$ is then a right Quillen equivalence. Finally the model structure on $\S\otimes\X$ is left proper and is right proper if $\S^{\Xi\op}$ is right proper.
\end{theo}

\begin{proof}
We denote by $\S^{\Xi\op}_{proj}$ the projective model structure on $\S^{\Xi\op}$ and by $\S\otimes\X_{proj}$ the model structure constructed in Proposition \ref{prop: projective model structure}. By assumption, the model structure $\S^{\Xi\op}$ is obtained by localizing $\S^{\Xi\op}_{proj}$ with respect to a set $U$ of maps in $\S^{\Xi\op}$.  Let $V$ be the set of maps of $\S\otimes\X$ of the form $\L(\id\otimes S)(u)$ for $u\in U$. Since $\S\otimes\X_{proj}$ is left proper by \ref{prop: left proper} and combinatorial, the left Bousfield localization $L_V(\S\otimes\X_{proj})$ of $\S\otimes\X_{proj}$ exists.

According to \cite[Theorem 3.3.20]{hirschhornmodel} and Proposition \ref{prop: Quillen equivalence projective}, there is a Quillen equivalence
\[(\id\otimes S):L_U(\S^{\Xi\op}_{proj})\leftrightarrows L_{V}(\S\otimes\X_{proj}):(\id\otimes N),\]
where the right adjoint preserves and reflects weak equivalences by \cite[Proposition 3.7.]{horelmodel}. 

Cofibrations are not changed by a left Bousfield localization, hence, by Proposition, \ref{prop:characterization of cofibrations}, a map is a cofibration if and only if its image by $(\id\otimes N)$ is one. It remains to show that $\id\otimes N$ reflects fibrations. Let $p:U\to V$ be a map in $L_V(\S\otimes\X_{proj})$ that is sent to a fibration by $(\id\otimes N)$. Let $i:A\to B$ be a trivial cofibration in $L_V(\S\otimes\X_{proj})$. Let
\[
\xymatrix{
A\ar[d]_{i}\ar[r]&U\ar[d]^{p}\\
B\ar[r]& V
}
\]
be a commutative diagram. We want to produce a lift. We can hit this diagram with $\id\otimes N$. The map $\id\otimes N(i)$ is a trivial cofibration. It follows that there exists a map $l:\id\otimes N(B)\to\id\otimes N(U)$ making the diagram
\[
\xymatrix{
\id\otimes N(A)\ar[d]_{\id\otimes N(i)}\ar[r]&\id\otimes N(U)\ar[d]^{\id\otimes N(p)}\\
\id\otimes N(B)\ar[r]\ar[ur]_{l}& \id\otimes N(V)
}
\]
commute. Since $\id\otimes N$ is fully faithful, this implies the existence of a lift for the initial square.

Finally, the model structure is left proper as is any left Bousfield localization of a left proper model structure and the argument for right properness is similar to the one in Proposition \ref{prop: left proper}.
\end{proof}

\section{Enrichment}

Let $\cat{V}$ be a cofibrantly generated monoidal model category. In this section, we study the existence of the structure of a $\cat{V}$-enriched model category on the model categories of Theorem \ref{theo: main}.

Let $\cat{C}$ be a bicomplete $\cat{V}$-enriched category that we assume to be tensored and cotensored over $\cat{V}$. For every $C\in \cat{C}$ and $V\in \cat{V}$ we denote by , $V\odot C$ and $C^{V}$  the tensor of $c$ with $v$ and the cotensor of $C$ with $V$ respectively, while the hom-functor will be denoted by $\emap{\cat{V}}(-,-)$. Recall that $\cat{C}$ is a $\cat{V}$-enriched model category if it has a model structure such that for every cofibration $i\colon K \to L$ in $\cat{V}$ and every cofibration $j\colon A \to B$ in $\cat{C}$ the induced pushout-product map
\[
i\square j \colon L\odot A \underset{K \odot A}\sqcup K\odot B \longrightarrow L\odot B
\] 
is a cofibration in $\cat{C}$ which is moreover a weak equivalence if either $i$ or $j$ are trivial cofibrations.

For every small category $D$ the category $\Fcat{D}{\cat{C}}$ is also $\cat{V}$-enriched and bicomplete in a natural way (see \cite[4.4]{GuillouMay}). Now, we suppose that our category of ``spaces'' $\S$, satisfying the assumptions \ref{assumption S}, is enriched over $\cat{V}$, with an enrichment that gives it the structure of a $\cat{V}$-model category; in particular $\S$ is tensored and cotensored over $\cat{V}$. For every locally presentable category $\cat{Y}$, that we can suppose isomorphic to $\Indl{\lambda}(A)$ for some small category $A$, the category $\cat{Y} \otimes \S$, seen as a full subcategory of $\Fcat{A \op}{\S}$, is naturally enriched in $\cat{V}$. It is also  bicomplete as a $\cat{V}$-category. For every $V \in \cat{V}$ the tensor with $V$ is given by $\id \otimes (V\odot -)$ and the cotensor with $V$ is given by $\id \otimes (-^V)$.  
 
Given another locally presentable category $\cat{Z}$, every adjunction $u^*:\cat{Y}\rightleftarrows\cat{Z}:u_*$ induces a $\cat{V}$-adjunction
\[
u^*\otimes \id :\cat{Y}\otimes \S \rightleftarrows\cat{Z} \otimes \S :u_*\otimes \id. 
\] 
Indeed, $u_*\otimes \id$ extends to a $\cat{V}$-functor and it preserves cotensors, thus it has a $\cat{V}$-left adjoint by \cite[Theorem 4.85]{KellyEC}. 

The following proposition can be seen as a particular case of \cite[Theorem 1.16]{GuillouMay}.

\begin{prop}
Let $\cat{V}$ be a cofibrantly generated monoidal model category  and suppose that $\S$ is $\cat{V}$-enriched as a model category. Then the model structure on $\S\otimes \X$ of Proposition \ref{prop: projective model structure} is $\cat{V}$-enriched.
\end{prop}	

\begin{proof}
	Recall that if $\S$ is a $\cat{V}$-model category then $\S^{\Xi\op}$ is naturally $\cat{V}$-enriched and the projective model structure is a $\cat{V}$-model structure (see \cite[4.4]{GuillouMay}).
	The adjunction $(\id\otimes S,\id\otimes N)$ can be extended to a $\cat{V}$-adjunction; since $\id\otimes N$ is right adjoint it preserves cotensors.
	
	To prove that $\S\otimes \X$ is a $\cat{V}$-model structure it is sufficient to check that for every cofibration $i\colon K \to L$ in  $\cat{V}$ and every fibration in $i\colon X \to Y$ in $\S\otimes \X$ the morphism
	\[
	 \coten{(i^*)}{p_*} \colon \coten{L}{X} \longrightarrow \coten{Y}{L} \underset{\coten{K}{Y}}\times \coten{K}{X}
	\]
	is a fibration, and a weak equivalence if $i$ or $p$ are weak equivalences (\emph{cf.} \cite[4.3]{GuillouMay}).
	Since $\id\otimes N$ creates (trivial) fibrations and preserves cotensors, the above statement follows from the fact that the projective model structure on $\S^{\Xi\op}$ is $\cat{V}$-enriched.   
\end{proof}

Following \cite{barwickleft}, we call a model category tractable if it is combinatorial and admits a set of generating cofibrations with cofibrant source. Let us assume that $\cat{V}$ is a tractable symmetric monodial model category. We recall that, given a $\cat{V}$-model structure $\cat{C}$ and a class of maps $H$ in $\cat{C}$, the $\cat{V}$-enriched left Bousfield localization of $\cat{C}$ at $H$, denoted by $L_{H/\cat{V}}(\cat{C})$, is the unique $\cat{V}$-enriched model structure on (the underlying $\cat{V}$-category of) $\cat{C}$ (if it exists) such that:
\begin{itemize}
    \item[-] The cofibrations are the cofibrations of $\cat{C}$.
    \item[-] An object $F$ is fibrant in $L_{H/\cat{V}}(\cat{C})$ if and only if it is fibrant in $\cat{C}$ and the map
    \[
    h_*\colon \rder\emap{\cat{V}}(B,F)
    \longrightarrow \rder\emap{\cat{V}}(A,F)
    \]
    is a weak equivalence (in $\cat{V}$) for every $h\in H$.
   \end{itemize} 

Barwick showed in \cite[Theorem 4.46]{barwickleft} that if $\cat{C}$ is a tractable left proper $\cat{V}$-model category and $H$ is a set of maps, the $\cat{V}$-enriched left Bousfield localization $L_{H/\cat{V}}(\cat{C})$
always exists (and it is tractable).

\begin{rem}
When $\cat{V}$ is the category $\Set^{\Delta\op}$ with the Kan-Quillen model structure, $\cat{V}$-enriched Bousfield localization at $H$ coincides with the usual Bousfield localization at $H$ and it is enough to assume that $\cat{C}$ is combinatorial and left proper for $L_{H/\sSet}(\cat{C})$ to exist (\emph{cf.}  \cite[Theorem 4.1.1]{hirschhornmodel}).
\end{rem}
     
Now, we assume that our category of spaces $\S$ is a model category satisfying assumption \ref{assumption S} which is furthermore a left proper tractable $\cat{V}$-model category. Under these assumptions $\S^{\Xi^{\op}}_{proj}$ is also tractable and left proper and the same holds for the model structure
$\S\otimes \X_{proj}$ of Proposition \ref{prop: projective model structure}. The following is an enriched version of Theorem \ref{theo: main}.

\begin{prop}\label{prop:enrichedmain}
Suppose $H$ is a set of maps in $\S^{\Xi^{\op}}$. The $\cat{V}$-model structure on $L_{H/\cat{V}} (\S^{\Xi^{\op}})$ can be transferred along $\id\otimes N$ to a $\cat{V}$-enriched model structure on $\S\otimes\X$ in which a map is a fibration (resp. cofibration, weak equivalence) if and only if its image by $\id\otimes N$ is a fibration (resp. cofibration, weak equivalence). The functor $\id\otimes N$ is then a right Quillen equivalence. Finally the model structure on $\S\otimes\X$ is left proper and is right proper if $\S$ is right proper.
\end{prop}

\begin{proof}
The enriched left Bousfield localization $L_{H/\cat{V}} (\S^{\Xi^{\op}})$ is the usual Bousfield localization of the projective model structure on $\S^{\Xi^{\op}}$ with respect to 
\[H\square I=\{h\square i\mid h\in H, i\in I\},\] 
where $I$ is a set of generating cofibrations for $\cat{V}$ (Theorem 4.46, \cite{barwickleft}); the statement follows from Theorem \ref{theo: main}. The transfer of this model structure on $\S\otimes \X$ is the localization of $\S\otimes\X_{proj}$ with respect to $\lder(\id\otimes S)(H\square I)=\lder (\id\otimes S)(H)\square I$ and it is thus the $\cat{V}$-enriched Bousfield localization $L_{\lder (\id\otimes S) H/\cat{V}} (\S^{\Xi^{\op}})$.
\end{proof}

\section{Applications}\label{sec:applications}

In this last section, we study some of the consequences of Theorem \ref{theo: main}. In the first subsection, we compare strict models for a fully faithful (\ref{defi: fully faithful}) finite connected limit sketch to homotopy models. Then we describe the consequences of this comparison to the the theory of operads, $n$-categories and $n$-fold categories. Finally in the last subsection, we explain how our work relates to the theory of monads with arities of \cite{bergermonads}.

\subsection{Models for a finite connected limit sketch}\label{sec:appconnectedsketch}

Let $(\cat{T},L)$ be a finite connected limit sketch and let $(\cat{A}\op,\theta(L))$ be the finite connected fully faithful sketch of Proposition \ref{prop:full subcategory sketch}. Note that if $(\cat{T},L)$ is already fully faithful, then $(\cat{A}\op,\theta(L))$ is isomorphic to $(\cat{T},L)$. Let $(S,N)$ be the nerve adjunction for $\cat{A}$, seen as a dense full subcategory of $\model{\cat{T},L}$. 

Let $\cat{S}$ be a model category satisfying the assumptions \ref{assumption S}. We have the category $\model{\cat{T},L}_\S$ of strict models for $(\cat{T},L)$ that are internal to $\S$ and the adjunction (\ref{eq:Ssketch}) decomposes as
\[
\xymatrix{\S^{\cat{T}} \ar@<3pt>[r]^-{\theta_!} & \ar@<3pt>[l]^-{\theta^*} \S^{\cat{A}\op} \ar@<3pt>[r]^-{S\otimes\id} & \ar@<3pt>[l]^-{N\otimes\id} \model{\cat{T},L}_{\S}.} 
\]
We then have the following proposition.

\begin{prop}
The projective model structure on $\S^{\cat{T}}$ can be transferred along the nerve functor 
\[i\otimes\id:\model{\cat{T},L}_\S\to\S^{\cat{T}}.\]
Moreover, if $(\cat{T},L)$ is a fully faithful sketch the functor $N\otimes\id$ becomes a Quillen equivalence.
\end{prop}

\begin{proof}
Since the functor $\theta \colon \cat{T}\to \cat{A}\op$ is a bijection on objects, the projective model structure on $\S^{\cat{T}}$ transfers along $\theta^*$ to the projective model structure on $\S^{\cat{A}\op}$. Since $\cat{A}$ is a dense full subcategory of $\model{\cat{T},L}$
the projective model structure transfers along $N\otimes \id$ by Theorem \ref{theo: main} and $N\otimes \id$ is a Quillen equivalence. It follows that the projective model structure on $\S^{\cat{T}}$ transfers along $l\otimes \id$; if $(\cat{T},L)$ is a fully faithful sketch $\cat{T}$ is isomorphic to $\cat{A}\op$, thus $l\otimes \id$ is a Quillen equivalence.
\end{proof}

Now, we want to localize the model structure on $\S^{\cat{T}}$ to obtain a new one in which the fibrant objects are a suitable notions of homotopy models for $(\cat{T},L)$. Roughly speaking a homotopy model for a limit sketch $(\cat{T},L)$ is a diagram $\cat{T}\to\S$ that sends the cones of $L$ to homotopy limit cones. For $\cat{M}$ a model category, we denote by $\R\Map_{\cat{M}}(X,Y)$ the derived mapping space from $X$ to $Y$. If $\cat{M}$ is a simplicial model category, this is just given by the ordinary mapping space from a cofibrant replacement of $X$ to a fibrant replacement of $Y$. In general, this can be constructed as the mapping space in the Hammock localization or via simplicial/cosimplicial framings.

\begin{lemm}\label{lemm: generators of a combinatorial model category}
Let $\cat{M}$ be a combinatorial model category. There exists a set $\mathscr{G}$ of objects of $\cat{M}$ such that a map $f:X\to Y$ is a weak equivalence of $\cat{M}$ if and only if for any $G$ in $\mathscr{G}$, the induced map
\[\R\Map_\cat{M}(G,X)\to \R\Map_{\cat{M}}(G,Y)\]
is a weak equivalence.
\end{lemm}

\begin{proof} 
By a result of Dugger (\emph{cf.} \cite[Propositions 3.2,3.3]{dugger2001combinatorial}) there exists a small category $A$ and a Quillen equivalence
\[
 \xymatrix{
%0-th row
(\sSet)^{A\op}_S %(0,0)
  \ar@<3pt>[r]^-{L} &
\cat{M} %(0,1)
  \ar@<3pt>[l]^-{R}
},
\]
where $(\sSet)^{A\op}_S$ is a left Bousfield localization of the projective model structure on the category $(\sSet)^{A\op}_{proj}$ ($\sSet$ is taken with the Kan-Quillen model structure).  
Let $\mathscr{G}=\{L\circ y(a)\mid a\in A\}$ be the image of the representable in $A$ under $L$.
For every $X\in \cat{M}$ and $a\in A$ the derived mapping space $\R\Map_{\cat{M}}(L\circ y(a),X)$ is weakly equivalent to $\R\Map_{(\sSet)^{A\op}_S}(y(a),\R R(X))$.
The statement follows from the fact that in $(\sSet)^{A\op}_S$ a map $f\colon X' \to Y'$ is a weak equivalence if and only if the induced map
\[\R\Map_{(\sSet)^{A\op}_S}(y(a),X')\to \R\Map_{(\sSet)^{A\op}_S}(y(a),Y')\]
is a weak equivalence for every $a\in A$.
\end{proof}

Assume now that $\S$ still satisfies the assumptions \ref{assumption S} and is moreover left proper. Let $\mathscr{G}$ be a set of homotopy generators of $\S$ whose existence is given by Lemma \ref{lemm: generators of a combinatorial model category}. Without loss of generality, we may assume that the elements of $\mathscr{G}$ are cofibrant. Let $l:I^{\triangleleft}\to\cat{T}$ be an element of $L$. We can then consider the composite
\[(I\op)^{\triangleright}\xrightarrow{l\op}\cat{T}\op\xrightarrow{y}\cat{Set}^{\cat{T}},\]
where $y$ is the Yoneda embedding. This cocone induces a map
\[f_{l,G}:\hocolim_{i\in I} G\otimes y(l(i))\to G\otimes y(l(\ast)),\]
where $\ast$ denotes the cone point of $I^{\triangleleft}$ and the homotopy colimit is computed in the model category $\S^{\cat{T}}_{proj}$. Let $H_{L,\mathscr{G}}=\{f_{l,G}\mid l\in L, G\in \mathscr{G}\}$ be the set of all such maps. We define $\hmodel{\cat{T},L}_\S$ to be the left Bousfield localization of $\S^{\cat{T}}_{proj}$ with respect to the maps of $H_{L,\mathscr{G}}$. We have the following characterization of the fibrant objects of $\hmodel{\cat{T},L}_\S$.

\begin{prop}\label{prop: characterization of fibrant objects}
An object $X$ in $\S^{\cat{T}}$ is fibrant in  $\hmodel{\cat{T},L}_\S$ if and only if it is projectively fibrant and for each $l\in L$, the map
\[X(l(\ast))\to \holim_{i\in I}X(l(i))\]
is a weak equivalence.
\end{prop}

\begin{proof}
Since the localization only depends on the homotopy type of the maps of $H_{L,\mathscr{G}}$, we may assume, without loss of generality, that the sources of all the maps of $H_{L,\mathscr{G}}$ are cofibrant (note that their targets are already cofibrant). Assume that $X$ is fibrant in  $\S^{\cat{T}}_{proj}$. Let $X_{\bullet}$ be a simplicial framing of $X$ in $\S^{\cat{T}}_{proj}$. Then, for any cofibrant object $A$ of $\S^{\cat{T}}_{proj}$, the simplicial set $\S^{\cat{T}}(A,X_{\bullet})$ is a model for the derived mapping space $\R\Map_{\S^{\cat{T}}}(A,X)$. Moreover, by definition of the projective model structure, for each $t$ in $\cat{T}$, the simplicial object $X_{\bullet}(t)$ is a simplicial framing of $X(t)$.

The object $X$ is fibrant in  $\hmodel{\cat{T},L}_\S$, if and only if for any $l\in L$ and $G\in \mathscr{G}$, the map
\[\S^{\cat{T}}(G\otimes y(l(\ast)),X_\bullet)\to \S^{\cat{T}}(\hocolim_{i\in I}G\otimes y(l(i)),X_\bullet)\]
is a weak equivalence. Equivalently, this happens if and only if the map
\[\S^{\cat{T}}(G\otimes y(l(\ast)),X_\bullet)\to \holim_{i\in I}\S^{\cat{T}}(G\otimes y(l(i)),X_\bullet)\]
is  a weak equivalence. Equivalently, by definition of the tensor product (see Equation \ref{c otimes d}), this happens if and only if we have a weak equivalence
\[\S(G,X_\bullet(l(\ast)))\to \holim_{i\in I}\S^{\cat{T}}(G,X_\bullet(l(i))).\]
Since this has to be true for any $G$, by Lemma \ref{lemm: generators of a combinatorial model category}, we see that $X$ is fibrant if and only if
\[X(l(\ast))\to \holim_{i\in I}X(l(i))\]
is a weak equivalence.
\end{proof}

%We define $\hmodel{\cat{T},L}_\S$ to be the $\S$-enriched left Bousfield localization $L_{H_L/\S}(\S^\cat{T}_{proj})$; it has the same cofibrations as $\S^{\cat{T}}$ and its fibrant objects are the fibrant objects $F$ of $\S^{\cat{T}}$ that are such that the map
%\[
%(f_{l})_* \colon F(l(\ast))\cong \rder\emap{\S}(y(l(\ast)),F)) \longrightarrow 
%\rder\emap{\S}(\mathrm{hocolim}_{i\in I}y(l(i)),F) \simeq \mathrm{holim}_{i\in I}F(i)
%\]
%is a weak equivalence for every $f_l\in H_L$. Such objects can reasonably be called Segal models for $(\cat{T},L)$ with values in $\S$.

%The model on $\S^{\cat{A}\op}$ transferred from $\hmodel{\cat{T},L}_\S$ along $\theta^*$ coincides with $L_{\L\theta_!(H_L)/\S}(\S^{\cat{A}\op}_{proj})=\hmodel{\cat{A}\op,L}_\S$. Thanks to Proposition \ref{prop:enrichedmain} the model structure  $\hmodel{\cat{A}\op,L}_\S$ can be transferred to $\S\otimes \cat{X}$ to produce a model structure $\S\otimes{\X}\segal$; furthermore, we have the following diagram of Quillen adjunctions:
%\[
%\xymatrix{\hmodel{\cat{T},L}_\S \ar@<3pt>[r]^-{\theta_!} & \ar@<3pt>[l]^-{\theta^*} \hmodel{\cat{A}\op,\theta(L)}_\S \ar@<3pt>[r]^-{S\otimes\id} & \ar@<3pt>[l]^-{N\otimes\id} {\S}\otimes{\X}\segal} 
%\]
%where $(S\otimes \id,N\otimes \id)$ is a Quillen equivalence; in particular if %$(T,L)_{\S}$ is a fully-faithful finite connected limit sketch (see \ref{defi: %fully faithful}), the adjunction $\hmodel{\cat{T},L}$ is Quillen equivalent to 
%${\S}\otimes{\X}\segal$ (via $l\otimes \id$).
%\commentc{When is $(\theta^*,\theta_*)$ a Quillen equivalence?} 

Our main theorem takes the following form.

\begin{theo}
Let $\S$ be a left proper model category satisfying the assumptions \ref{assumption S}. Let $(\cat{T},L)$ be a finite connected limit sketch which is fully faithful in its category of models. Then there is a model structure on $\model{\cat{T},L}_\S$ in which a map is a fibration (resp. cofibration, weak equivalence) if and only if its image by $\id\otimes N$ is a fibration (resp. cofibration, weak equivalence) in $\hmodel{\cat{T},L}_\S$. The functor $\id\otimes N$ is then a right Quillen equivalence. Finally the model structure on $\model{\cat{T},L}_\S$ is left proper.
\end{theo}

\begin{proof}
Follows immediately from Theorem \ref{theo: main}.
\end{proof}

\subsection{Operads}

We denote by $\cat{Op}$ the category of small multicategories in sets. As explained in \cite{moerdijkdendroidal}, there exists a category $\Omega$ of trees equipped with a fully faithful functor $\Omega\to\cat{Op}$ inducing a nerve functor $N_d:\cat{Op}\to \Set^{\Omega\op}$. It is straightforward to see that the nerve $N_d$ preserves filtered colimits and coproducts. Thus, the pair $(\cat{Op},\Omega)$ satisfies the hypothesis of \ref{assumptions X}. 

We can form the projective model structure on $\S^{\Omega\op}$. Let $\mathscr{G}$ be a set of generators as in \ref{lemm: generators of a combinatorial model category}. Assuming that $\S$ is left proper we can perform the $\S$-enriched left Bousfield localization of $\S^{\Omega\op}$ with respect to the maps
\begin{equation}\label{eq:segaloper}
 G\otimes\mathrm{Sc}(T)\to G\otimes \Omega[T]
\end{equation}
for any tree $T$ in $\Omega$ and any $G\in\mathscr{G}$. We denote this model structure by $\S^{\Omega\op}_{Segal}$. We prove exactly as in Proposition \ref{prop: characterization of fibrant objects} that its fibrant objects are the functors $\Omega\op\to\S$ satisfying the Segal condition. We immediately deduce the following theorem from Theorem \ref{theo: main}.

\begin{theo}\label{theo:main operads}
The projective and the Segal model structure on $\S^{\Omega\op}$ can be transferred to a model structure on $\S\otimes\cat{Op}$ along the nerve functor $\id\otimes N_d$. Moreover, in both cases, the functor $\id\otimes N_d$ is a right Quillen equivalence and preserves and reflects cofibrations.
\end{theo}

Note that if $\S$ is the category $\sSet$ with its usual model structure, the model category $\S^{\Omega\op}_{Segal}$ is Quillen equivalent to the model structure of dendroidal Segal spaces constructed in \cite[Definition 5.4.]{cisinskidendroidal}. In fact, one can show using \cite[Proposition 1.7.]{horelmodel} that the two model structures have the same weak equivalences. In that case, we can construct a further localization $\S^{\Omega\op}_{Rezk}$ of $\S^{\Omega\op}$ in which the fibrant objects are the fibrant in $\S^{\Omega\op}_{Segal}$ that are moreover complete. Theorem \ref{theo: main} allows us to transfer this model structure to a Quillen equivalent model structure on $\S\otimes\cat{Op}$ denoted $\S\otimes\cat{Op}_{Rezk}$. Moreover the model structure $\S^{\Omega\op}_{Rezk}$ is itself Quillen equivalent to the model structure constructed in \cite[Definition 6.2.]{cisinskidendroidal}. In particular, the $\infty$-category underlying $\S\otimes\cat{Op}_{Rezk}$ is a model for the $\infty$-category of $\infty$-operads.

\subsection{\texorpdfstring{$n$}{n}-categories}\label{sec:ncategories}

The category of $n$-categories denoted $n\Cat$ is defined inductively. The induction is started by defining $1\Cat=\Cat$. Then, the category $n\Cat$ is the category of categories enriched in $(n-1)\Cat$. There is a full subcategory $\Theta_n$ (\emph{cf.} \cite[Corollary 3.10]{rezkcartesian}) of $n\Cat$ spanned by compact connected objects. It follows that the pair $(n\Cat,\Theta_n)$ satisfies the assumptions of \ref{assumptions X}. Moreover, Rezk constructs inductively a set $\mathscr{T}_{n}^{Se}$ of maps in $\Set^{\Theta_n\op}$ such that the $\Theta_n$-spaces that are local with respect to these maps satisfy the suitable form of Segal condition that makes them models for weak $n$-categories. The construction of this set is inductive and rather involved. We refer the reader to Section 5 of \cite{rezkcartesian}.

For $\S$ a model category satisfying the assumptions \ref{assumption S} and left proper with set of generators $\mathscr{G}$ (see Lemma \ref{lemm: generators of a combinatorial model category}) we can form the left Bousfield localization of $\S^{\Theta_n\op}$ with respect to the set $\bigsqcup_{G\in\mathscr{G}}G\otimes\mathscr{T}_{n}^{Se}$. We denote this model structure by $\S^{\Theta_n\op}_{Segal}$ and call it the Segal model structure. Applying Proposition \ref{prop: characterization of fibrant objects}, we immediately see that the fibrant objects in this model structure are the functors $\Theta_n\op\to \S$ satisfying the Segal condition.

Theorem \ref{theo: main} immediately yields the following.

\begin{theo}\label{theo:main n-cat}
The projective and the Segal model structure on $\S^{\Theta_n\op}$ can be transferred to a model structure on $\S\otimes n\Cat$ along the nerve functor $\id\otimes N$. Moreover, in both cases, the functor $\id\otimes N$ is a right Quillen equivalence and preserves and reflects cofibrations.
\end{theo}

If $\S$ is the model category of simplicial sets with the Kan-Quillen model structure, then we can construct a further localization of $\S^{\Theta_n\op}$ in which the fibrant objects are fibrant in $\S^{\Theta_n\op}_{Segal}$ and if moreover they are complete. The set of maps with respect to which we need to localize is the set $\mathscr{T}_{n,\infty}$ defined in \cite[11.4.]{rezkmodel}. The resulting model structure denoted $\S^{\Theta_n\op}_{Rezk}$ is Quillen equivalent to the model structure of $(\infty,n)$-$\Theta$-spaces defined in \cite[11.5.]{rezkmodel}. Theorem \ref{theo: main} allows us to transfer this model structure to a Quillen equivalent model structure on $\S\otimes n\Cat$ that we denote $\S\otimes n\Cat_{Rezk}$. In particular, the $\infty$-category underlying $\S\otimes n\Cat_{Rezk}$ is a model for the $\infty$-category of $(\infty,n)$-categories in the sense that it satisfies the axiomatic description of \cite{barwickunicity}.

\subsection{\texorpdfstring{$n$}{n}-fold categories}

The category of $n$-fold category is the category $\Cat^{\otimes n}$ where $\Cat$ denotes the locally presentable category of small categories. The category $\Cat$ can be expressed as a reflective subcategory of $\Set^{\Delta\op}$ via the usual nerve functor $\Cat\to\Set^{\Delta\op}$. It follows that $\Cat^{\otimes n}$ is a reflective subcategory of $\Set^{(\Delta^n)\op}$, moreover, one easily verifies that the image of $\Delta^n$ in $\Cat^{\otimes n}$ consists of compact and connected objects. It follows that the pair $(\Cat^{\otimes n},\Delta^n)$ satisfies the assumptions of \ref{assumptions X}.

We denote by $\Delta[k_1,k_2,\ldots,k_n]$ the objects of $\Set^{(\Delta^n)\op}$ represented by $[k_1]\times\ldots\times[k_n]$. Note that the object $\Delta[k_1,\ldots,k_n]$ is isomorphic to $\Delta[k_1]\otimes\ldots\otimes\Delta[k_n]$ modulo the equivalence of categories $\Set^{(\Delta^n)\op}\simeq(\Set^{\Delta\op})^{\otimes n}$. 

We denote by $\mathscr{S}$ the set of arrows in $\Set^{\Delta\op}$
\[f(n):G(n)\to\Delta[n],\]
where $G(n)$ is the object of $\Set^{\Delta\op}$ representing the $n$-fold fiber product $X\mapsto X_1\times_{X_0}X_1\times_{X_0}\ldots\times_{X_0}X_1$ and $f(n)$ represents the Segal map
\[X_n\to  X_1\times_{X_0}X_1\times_{X_0}\ldots\times_{X_0}X_1.\]
A simplicial set is local with respect to the maps of $\mathscr{S}$ if and only if it is in the essential image of the nerve functor $\Cat\to\Set^{\Delta\op}$. 

We denote by $\mathscr{S}^{\otimes n}$ the smallest set of arrows of $\Set^{(\Delta^n)\op}$ containing the arrows $f(k_1)\otimes\id_{\Delta[k_2]}\otimes\ldots \otimes\id_{\Delta[k_n]}$ for all $n$-tuple $(k_1,\ldots,k_n)$ and which is invariant under the action of $\Sigma_n$ on $\Set^{(\Delta^n)\op}$. 

As in the previous paragraph, we assume that $\S$ is left proper and we denote by $\mathscr{G}$ a set of homotopy generators (see Lemma \ref{lemm: generators of a combinatorial model category}). We denote by $\S^{(\Delta^n)\op}_{Segal}$ the left Bousfield localization of the projective model structure on $\S^{(\Delta^n)\op}$ with respect to the maps in $\bigsqcup_{G\in\mathscr{G}}G\otimes\mathscr{S}^{\otimes n}$. Using Propostion \ref{prop: characterization of fibrant objects}, we see that an object is fibrant in $\S^{(\Delta^n)\op}_{Segal}$ if it is projectively fibrant and if the simplicial object of $\S$ obtained by fixing all the variables but one is fibrant in $\S^{\Delta\op}_{Segal}$. Theorem \ref{theo: main} then implies the following theorem.

\begin{theo}\label{theo: main n-fold}
The projective and the Segal model structure on $\S^{(\Delta^n)\op}$ can be transferred to a model structure on $\S\otimes\cat{Cat}^{\otimes n}$ along the nerve functor $\id\otimes N$.Moreover, in both cases, the functor $\id\otimes N_d$ is a right Quillen equivalence and preserves and reflects cofibrations.
\end{theo}

\begin{example}
The authors of \cite{calaquenote} construct a $\Delta^n$-space encoding the cobordism $n$-fold $\infty$-category. Using our main theorem \ref{theo: main n-fold}, we can rigidify this model to an equivalent strict $n$-fold category in the category of simplicial sets with the same homotopy type. An explicit construction of a strict $n$-fold category of cobordisms is also done in \cite{bokstedtcobordism}.
\end{example}

\subsection{Monads with (connected compact) arities}

Given a monad $T$ on a category $\cat{B}$ we denote its category of algebras by $\Alm{T}$. It is known that if $\cat{B}$ is finitely locally presentable and $T$ is finitary, the category $\Alm{T}$ is also finitely locally presentable and thus presented by a finite limit sketch. However, in applications, one would like to have an explicit description of a sketch describing the theory considered. 

When $\cat{B}$ is a presheaf category, the theory of monads with arities introduced by Berger, Melli\`es and Weber (\emph{cf.} \cite{weber2007familial,bergermonads}) provides a way to recover a fully faithful sketch describing $\Alm{T}$ from the data of $T$ and a sufficiently nice dense subcategory $\Theta_0$ of $\cat{B}$. The aim of this section is to recall this theory and to show how our theorem applies in this situation.

We fix a small category $C$ and a monad $T$ over $\Set^{C\op}$. Given a dense full subcategory $\Theta_0$ of $\Set^{C\op}$ and an object $A\in \Set^{C\op}$, the $\Theta_0$-cocone over $A$ is just the canonical cocone $(\Theta_0\downarrow A)^{\triangleright} \to \Set^{C\op}$ with tip $A$.
The monad $T$ is \textbf{with arities $\Theta_0$} if the composite $N_{\Theta_0} T$ sends $\Theta_0$-cocones in $\Set^{C\op}$ to colimits-cocones in $\Set^{\Theta_0\op}$ (see  \cite[Definition 1.8]{bergermonads}). 

%Let $\Theta_0$ be a dense full subcategory of $\Set^{C\op}$ that provides arities for $T$ and let $N_{\Theta_0}\colon \Set^{C\op}\hspace{-3pt} \to \Set^{\Theta_0\op}$ be the associated nerve functor.
Let $\Theta_0$ be a dense full subcategory of $\Set^{C\op}$ that provides arities for $T$ and denote the associated nerve functor by $N_{\Theta_0}\colon \Set^{C\op}\hspace{-3pt} \to \Set^{\Theta_0\op}$.
If we denote by $\Theta_T$ the full image of $\Theta_0$ in $\Alm{T}$, with nerve functor $N_{\Theta_T}$, we have the following  diagram:
\begin{equation}\label{eq:aritiesdiag1}
 \xymatrix{
%0-th row
\Theta_T %(0,0)
  \ar[r]^-{} &
\Alm{T} %(0,1)
  \ar@<3pt>[d]^-{U} 
  \ar[r]^-{N_{\Theta_T}} &
\Set^{\Theta_T\op}
  \ar@<3pt>[d]^-{j^*}
\\%1-th row
\Theta_0 %(1,0)
  \ar[u]^-{j}
  \ar[r]^-{} &
\Set^{C\op} %(1,1)
  \ar@<3pt>[u]^-{F}
  \ar[r]^-{N_{\Theta_0}} &
  %\ar@<3pt>[l]^-{}
\Set^{\Theta_0\op}.
  \ar@<3pt>[u]^-{j_!}}
\end{equation}
Berger, Melli\`es and Weber showed that if $T$ is with arities $\Theta_0$ then $\Theta_T$ is dense in $\Alm{T}$ and the essential image of $N_{\Theta_T}$ is spanned by the presheaves $F$ such that $j^*(F)$ is in the essential image of $N_{\Theta_0}$ (Theorem 1.10 \cite{bergermonads}, Theorem 4.10 \cite{weber2007familial}). 

Since we are interested in finite connected sketches we restrict our attention to a more specific situation:

\begin{assu}\label{ass:monadwithar}
Let $C$, $\Theta_0$ and $T$ as above, we make the following assumptions:
\begin{enumerate}
 \item $T$ is a monad with arities $\Theta_0$,
 \item $\Theta_0$ contains $C$,
 \item for every $a\in \Theta_0$ the comma category $C\downarrow a$ is finite and connected. This implies in particular that all the objects of $\Theta_0$ are compact and connected in $\Set^{C\op}$.
\end{enumerate}
\end{assu}
For every $a\in \Theta_0$ let $l_a\colon (C\downarrow a)^{\triangleright} \to \Theta_0$ be the canonical $C$-cocone over $a$ and let
\[L=\{l_a\op \colon (a \downarrow C\op)^{\triangleleft} \to \Theta_0\op\mid a\in \Theta_0\}\] the set of all the opposite cones. 

The maps in $\Set^{\Theta_0\op}$ of the kind
\[
 a\longrightarrow \underset{c\in C\downarrow a} \colim h_c
\]
for some $a\in \Theta_0$ are called the \textbf{Segal maps}.

The essential image of $N_{\Theta_0}$ is spanned by those $F\in\Set^{\Theta_0\op}$ such that the canonical map induced by $l_a$:
\[
 F(a)\cong \Set^{\Theta_0\op}(a,F)\to \Set^{\Theta_0\op}(\underset{c\in C\downarrow a} \colim \,c,F)\cong  \underset{c\in C\downarrow a}\lim F(c)
\]
is an isomorphism for every $a\in \Theta_0$ (\emph{cf.} Proposition 4.13 \cite{weber2007familial}); 
in other words $\Set^{C\op}$ is equivalent to the category of models of the finite connected limit sketch $(\Theta_0\op,L)$.  

Therefore diagram (\ref{eq:aritiesdiag1}) above is equivalent to
\begin{equation}\label{eq:aritiesdiag2}
 \xymatrix{
%0-th row
\Theta_T %(0,0)
  \ar[r]^-{} &
\model{\Theta_T\op,F(L)}_{\Set} %(0,1)
  \ar@<3pt>[d]^-{U} 
  \ar@<3pt>[r]^-{N_T} &
\Set^{\Theta_T\op}
  \ar@<-3pt>[d]_-{j^*}
  \ar@<3pt>[l]^-{S_T}
\\%1-th row
\Theta_0 %(1,0)
  \ar[u]^-{j}
  \ar[r]^-{} &
\model{\Theta_0\op,L}_{\Set} %(1,1)
  \ar@<3pt>[u]^-{F}
  \ar@<3pt>[r]^-{N_{\Theta_0}} &
  %\ar@<3pt>[l]^-{}
\Set^{\Theta_0\op}
   \ar@<3pt>[l]^-{S_T}
   \ar@<-3pt>[u]_-{j_!}}.
\end{equation}

In particular the category $\Alm{T}$ can be described as the category of models for the fully faithful finite connected limit sketch $(\Theta,j(L))$. 

If we let $\S$ to be as in Section \ref{sec:appconnectedsketch}, tensoring with $\S$ we get the following diagram of Quillen adjunctions
\begin{equation}
 \xymatrix{
%0-th row
%\Theta_T %(0,0)
%  \ar[r]^-{} &
\model{\Theta_T\op,F(L)}_{\S} %(0,1)
  \ar@<3pt>[d]^-{U} 
  \ar@<3pt>[r]^-{N_T} &
\S^{\Theta_T\op}_{proj}
  \ar@<-3pt>[d]_-{j^*}
  \ar@<3pt>[l]^-{S_T}
\\%1-th row
%\Theta_0 %(1,0)
%  \ar[u]^-{}
%  \ar[r]^-{} &
\S^{C\op}\cong\model{\Theta_0\op,L}_{\S} %(1,1)
  \ar@<3pt>[u]^-{F}
  \ar@<3pt>[r]^-{N_{\Theta_0}} &
  %\ar@<3pt>[l]^-{}
\S^{\Theta_0\op}_{proj}
  \ar@<-3pt>[u]_-{j_!}
  \ar@<3pt>[l]^-{S_{\Theta_0}}}
\end{equation}
where the horizontal adjunctions are Quillen equivalences. Localizing at the Segal maps as in Section \ref{sec:appconnectedsketch} on both rows, the diagram remains a diagram of Quillen adjunctions
\begin{equation}
 \xymatrix{
%0-th row
%\Theta_T %(0,0)
%  \ar[r]^-{} &
(\model{\Theta_T\op,j(L)}_{\S})\segal %(0,1)
  \ar@<3pt>[d]^-{U} 
  \ar@<3pt>[r]^-{N_T} &
\hmodel{\Theta_T\op,j(L)}_{\S}
  \ar@<-3pt>[d]_-{j^*}
  \ar@<3pt>[l]^-{S_T}
\\%1-th row
%\Theta_0 %(1,0)
%  \ar[u]^-{}
%  \ar[r]^-{} &
(\model{\Theta_0\op,L}_{\S})\segal %(1,1)
  \ar@<3pt>[u]^-{F}
  \ar@<3pt>[r]^-{N_{\Theta_0}} &
  %\ar@<3pt>[l]^-{}
\hmodel{\Theta_0\op,L}_{\S}
  \ar@<-3pt>[u]_-{j_!}
  \ar@<3pt>[l]^-{S_{\Theta_0}}}
\end{equation}
where the rows are Quillen equivalences.

All the examples presented in the previous sections can all be regarded as instances of the monad with arities, as explained in \cite{bergermonads}. For example $\Cat$ is the category of algebras for a monad $T$ over the category of graphs $\Set^{\mathbbm{2}\op}$, where $\mathbbm{2}$ is the category with set of objects $\{0,1\}$ and with only two arrows different from the identities, both from $0$ to $1$.

Let $\Theta_0$ full subcategory  of $\Set^{\mathbbm{2}\op}$ spanned by the finite linear graphs, i.e. those $G\in\Set^{\mathbbm{2}\op}$ such that $G(0)$ and $G(1)$ are finite, $|G(0)|=|G(1)|+1$, $G(s)$ and $G(t)$ are injective and with different images. $\Theta_0$  satisfies assumption \ref{ass:monadwithar} and is equivalent to $\Delta_{out}$, the wide subcategory of $\Delta$ spanned by all the maps $f$ between finite ordinals such that $f(i-1)=f(i)-1$ for every $i>0$ in the domain of $f$.
The full image of $\Theta_0$ in $\Cat$ is equivalent to $\Delta$ and the Segal maps correspond to the classical Segal maps.

Similarly, the dendroidal category $\Omega$ can be obtained from the monad over the category of \textbf{symmetric multi-graphs} which has $\cat{Op}$ as category of algebras and has the category of trees (and tree embeddings) as arities (\emph{cf.} \cite{weber2007familial} and \cite{kocktrees}). The Segal maps correspond exactly to the maps of type (\ref{eq:segaloper}).

The category of $n$-categories of section \ref{sec:ncategories} is the category of algebras for a monad $T$ over ($n$-truncated) globular sets; the category of $n$-globular pasting diagrams provides arities for $T$, satisfies assumption \ref{ass:monadwithar} and its full image in $n\cat{Cat}$ is isomorphic to $\Theta_n$.

The category $\omega \cat{Cat}$ of strict infinity categories is also the category of algebras for a monad $T$ over (non-truncated) globular sets. The category of globular pasting diagrams satisfies assumption \ref{ass:monadwithar} (with respect to $T$). We refer the reader to \cite[2.11,3.12]{bergermonads},\cite{bergeriterated} and \cite[8.1]{leinster2004higher} for details. 

Other examples of monads with arities are given in \cite{bergermonads} and \cite{weber2007familial}. One example of monad with arities that satisfies assumption \ref{ass:monadwithar} is the monad for properads over the presheaf category of \textbf{digraphical species} presented in \cite{kockhypergraph}, which has the category of connected acyclic graph and \'etale maps (presented in \emph{loc. cit.}) as arities.

%Weber also showed that if $T$ is a parametric right adjoint there is a canonical choice for the category of arities, called the \emph{category of $T$-cardinals} (\emph{cf.} Def. 4.16  \cite{weber2007familial}).

\bibliographystyle{gtart}
\bibliography{biblio}

\end{document}